\documentclass[12pt]{amsart}
\usepackage[alphabetic]{amsrefs}
\usepackage{mathdots, blkarray}
\usepackage{multicol}
\usepackage{graphicx}
\usepackage{amscd}
\usepackage{upgreek}
\usepackage{stmaryrd}
\usepackage{longtable}
\usepackage[T1]{fontenc}
\usepackage{latexsym, amsmath, amssymb, amsthm}
\usepackage[svgnames]{xcolor}
\usepackage{rsfso}
\usepackage[mathscr]{eucal}
\usepackage{mathtools}
\usepackage{mathptmx}
\usepackage{titletoc}
\usepackage{wrapfig}
\usepackage{float}
\usepackage{xypic}
\usepackage{microtype}
\usepackage{dsfont}
\usepackage{xcolor}
\usepackage{color}
\usepackage[colorlinks = true,
            linkcolor  = DarkBlue,
            urlcolor   = DarkRed,
            citecolor  = DarkGreen]{hyperref}
\usepackage{hyperref}
\allowdisplaybreaks	
\usepackage[centering, includeheadfoot, hmargin=1.0in, tmargin=0.5in,  bmargin=0.6in, headheight=6pt]{geometry}
  
\usepackage{orcidlink}

\newtheorem{theorem}{Theorem}[section]
\newtheorem{prop}[theorem]{Proposition}

\newtheorem{lem}[theorem]{Lemma}
\newtheorem{coro}[theorem]{Corollary}

\newtheorem{thm}[theorem]{Theorem}

\newtheorem{rem}[theorem]{\rm\textsc{Remark}}
\newtheorem{exam}[theorem]{\rm\textsc{Example}}
\newtheorem{nc}[theorem]{\rm \textsc{Notation}}

\newcommand{\bslash}{\kern-0.1em\texttt{\scalebox{0.6}[1]{/}}\kern-0.15em \texttt{\scalebox{0.6}[1]{/}}}

\newcommand{\FF}{\mathbb{F}}
\newcommand{\R}{ \mathcal{R}}

\newcommand{\OO}{O_{2}^{+}(\Fq)}
\newcommand{\Oo}{O_{2}^{-}(\Fq)}
\newcommand{\GL}{{\rm GL}_{2}(\Fq)}
\newcommand{\N}{\mathbb{N}}
\newcommand{\J}{\mathfrak{J}}
\newcommand{\Fq}{\mathbb{F}_{q}}

\newcommand{\FqmV}{\mathbb{F}_{q}[x_{1},x_{2},\dots,x_{m},y_{1},y_{2},\dots,y_{m}]}
\newcommand{\va}{(\upalpha_{1},\upalpha_{2},\dots,\upalpha_{m})}

\newcommand{\xa}{x_{1}^{\upalpha_{1}}x_{2}^{\upalpha_{2}}\cdots x_{m}^{\upalpha_{m}}}
\newcommand{\ya}{y_{1}^{\upalpha_{1}}y_{2}^{\upalpha_{2}}\cdots y_{m}^{\upalpha_{m}}}

\newcommand{\xn}{x_{1},x_{2},\dots,x_{n}}
\newcommand{\ma}{{\rm max}}
\newcommand{\IR}{\Fq[mV]^{\OO}}
\newcommand{\Ir}{\Fq[mV]^{\Oo}}
\newcommand{\mo}{{\rm mod~}}

\newcommand{\h}{\mathfrak{h}_{mV}(\OO)}
 
\newcommand{\bet}{\upbeta} 
\newcommand{\gam}{\upgamma} 
\newcommand{\del}{\updelta}

\newcommand{\sig}{\upsigma}

\newcommand{\hbo}{$\hfill\Diamond$}

\begin{document}
\title{Vector invariants for two-dimensional orthogonal groups\\ over finite fields} 
\def\shorttitle{Vector invariants for two-dimensional orthogonal groups over finite fields}

\author{Yin Chen \orcidlink{0000-0002-7573-3343}}
\address{School of Mathematics and Statistics, Northeast Normal University,
 Changchun 130024, China and (Current address) Department of Finance and Management Science, University of Saskatchewan, Saskatoon, SK, Canada, S7N 5A7}
\email{yin.chen@usask.ca}

\begin{abstract}
Let $\Fq$ be a finite field of characteristic 2 and $\OO$ be the $2$-dimensional orthogonal group over $\Fq$. Consider the standard representation $V$ of $\OO$ and the 
ring of vector invariants $\IR$ for any $m\in \N^{+}$. We prove a first main theorem for $(\OO,V)$, i.e., we find a minimal generating set  for $\IR$. As a consequence, we derive the Noether number $\bet_{mV}(\OO)=\ma\{q-1,m\}$. We construct a free basis for $\Fq[2V]^{\OO}$ over a suitably chosen homogeneous system of parameters. We also 
obtain a generating set for the Hilbert ideal of $\IR$ which shows that the Hilbert ideal can be generated by invariants of degree $\leqslant q-1=\frac{|\OO|}{2}$, confirming the conjecture of Derksen and Kemper \cite[Conjecture 3.8.6 (b)]{DK02} in this particular case.
\end{abstract}

\date{\today}
\maketitle \baselineskip=16pt

\dottedcontents{section}[1.16cm]{}{1.8em}{5pt}
\dottedcontents{subsection}[2.00cm]{}{2.7em}{5pt}

\section{Introduction}
\setcounter{equation}{0}
\renewcommand{\theequation}
{1.\arabic{equation}}
\setcounter{theorem}{0}
\renewcommand{\thetheorem}
{1.\arabic{theorem}}

\noindent  Let $\FF$ be a field, $G$ a finite group and $W$ be a faithful finite-dimensional representation of $G$ over $\FF$. The action of $G$ on $W$ induces a linear action  on the dual space $W^{\ast}$ by $\sig\cdot x=x\circ  \sig^{-1}$
for $\sig\in G$ and $x\in W^{\ast}$. Extending the action on $W^{\ast}$ multiplicatively yields an action of $G$ on $\FF[W]$,
the symmetric algebra on $W^{\ast}$. We choose $\{\xn\}$ as a basis of $W^{\ast}$. Then 
$\FF[W]$ can be identified with the polynomial ring $\FF[\xn]$. The subalgebra
$$\FF[W]^{G}:=\{f\in \FF[W]\mid \sig\cdot f=f,\textrm{ for all }\sig\in G\}$$
is called the \textit{invariant ring} of $G$ on $W$. 

Fix a representation $V$ of a finite group $G$ and consider $W=mV:=V\oplus V\oplus\cdots\oplus V$,  the direct sum of $m$ copies of $V$. Then $G$ acts on $W$ by extending diagonally the action on $V$.
Finding generators for the vector invariant ring $\FF[W]^{G}=\FF[mV]^{G}$ for a classical group $G$ (usually, $\FF$ is the field of complex numbers or the field of real numbers)  is the central  problem in classical invariant theory. 
According to H. Weyl's famous book \cite{Wey97}, a theorem giving a set of explicit generators for $\FF[mV]^{G}$ is referred to as a \textit{first main theorem} for $(G,V)$. 

The modular cases where the characteristic of $\FF$ divides the order of $G$ are more complicated.
In 1990, Richman \cite{Ric90} began the study of the vector invariants of $C_{p}$ acting on its two-dimensional indecomposable representation $V_{2}$ in characteristic $p>0$, giving a conjecture on generators for $\FF_{p}[mV_{2}]^{C_{p}}$ with a proof of the case $p=2$.
In 1997, Campbell and Hughes \cite{CH97} proved that Richman's conjecture was correct. In 2002, Shank and Wehlau \cite{SW02} gave a minimal generating set for $\mathbb{F}_{p}[mV_{2}]^{C_{p}}$. In 2010, Campbell, Shank and Wehlau \cite{CSW10} proved that the minimal generating set is actually a SAGBI basis for $\mathbb{F}_{p}[mV_{2}]^{C_{p}}$. In 2013, Wehlau \cite{Weh13} gave a new proof for Richman's conjecture via classical invariant theory. 
Recently, Bonnaf\'e and Kemper \cite{BK11}, Chen \cite{Che14, Che18}, Chen and Wehlau \cite{CW17}  also initiated a study of modular invariants of one vector and one covector for 
some linear groups over finite fields.

The present paper is devoted to study of the vector invariant ring for the two-dimensional orthogonal group of plus type over a finite field of characteristic $2$ with the standard representation.

The following theorem is our main result. 

\begin{thm} \label{T1}
Let $\Fq$ be a finite field of characteristic 2 and 
$\OO=\langle \upsigma,\uptau_{a}\rangle$
be the $2$-dimensional orthogonal group over $\Fq$ generated by
$$\sig=\begin{pmatrix}
      0 & 1 \\
     1  & 0\\
    \end{pmatrix}\textrm{ and }\uptau_{a}=\begin{pmatrix}
      a & 0 \\
     0  & a^{-1}\\
    \end{pmatrix},$$ where $a\in \Fq^{\times}.$ Suppose that  $\OO$
acts linearly on the polynomial ring
$$\Fq[mV]:=\FqmV$$  by 
$\sig(x_i)=y_i,\sig(y_i)=x_i$ and 
$\uptau_a(x_i)=a^{-1}\cdot x_i, \uptau_a(y_i)=a\cdot y_i
$
 for $1\leqslant i\leqslant m$. Then  
 $\IR$ is generated by 
\begin{eqnarray*}
\mathcal{N} &:= & \Big\{N_i=x_iy_i ~\big|~ 1\leqslant i\leqslant m\Big\}\\
\mathcal{U} &:=& \Big\{U_{ij}=x_iy_j+x_jy_i~\big|~  1\leqslant i<j\leqslant m\Big\}\\
\mathcal{B} &:=& \Big\{B_{\upalpha}=\xa+\ya ~\big|~  \upalpha\in\N^{m}, |\upalpha|=q-1\Big\}\\
\mathcal{D} &:=& \Big\{d_{I,J}=x_{I}\cdot y_{J}+y_{I}\cdot x_{J}~\big|~  \emptyset\neq I<J\subseteq \overline{m}, |J|-|I|=0 {\rm~~ or ~~}q-1\Big\},
\end{eqnarray*}  
where $|\upalpha|$, $d_{I,J}$ and $\overline{m}$ are defined as in   {\textsc{Notation \ref{nc}}} below. Moreover, $\IR$ is generated minimally by  $\mathcal{N} \cup \mathcal{B} \cup \mathcal{D}$.
\end{thm}

 \begin{rem} {\rm
Note that $\mathcal{U}$ is contained in $\mathcal{D}$. We pay a special attention to $U_{ij}$ because they will play an important role in our arguments below.  
\hbo}\end{rem}

Recall that the polynomial ring  $\FF[W]=\bigoplus_{d=0}^{\infty}\FF[W]_{d}$ is standard $\N$-graded and $G$ preserves the degrees. Thus
the invariant ring $\FF[W]^{G}=\bigoplus_{d=0}^{\infty}\FF[W]_{d}^{G}$ is also standard $\N$-graded. The number
$$\bet_{W}(G):=\textrm{min}\Big\{e~\big|~ \FF[W]^{G}\textrm{ is generated by }\bigoplus_{d=0}^{e}\FF[W]_{d}^{G}\Big\}$$
is called the \textit{Noether number} for $(G,W)$. As a   consequence of Theorem \ref{T1}, we derive

 \begin{coro} \label{noethernumber}
$\bet_{mV}(\OO)=\ma\{q-1,m\}$, for any $m\in \N^{+}$.
  \end{coro}

 \begin{rem} {\rm
It is worth noting that Symonds in \cite[Corollary 0.2]{Sym11} recently proved that  for any modular representation $W$ of a finite group $G$, $\bet_{W}(G)\leqslant\dim(W)(|G|-1)$. 
For other finite classical groups, we just know that \cite[Corollary 8.5]{CSW10} gives an upper bound for the Noether number $\bet_{mV_{2}}(\textrm{SL}_{2}(\FF_{p}))$, where $V_{2}$ is the standard representation of $\textrm{SL}_{2}(\FF_{p})$.
 \hbo}\end{rem}

 \begin{exam}\label{e1.5}
 {\rm ($m=2$)
Consider the set $\mathcal{D}$. Note that
$1\leqslant|I|\leqslant m-1$ and $1\leqslant|J|\leqslant m-1$.
In this case,  we must have $|I|=|J|=1$. Since $I<J$, then $I=\{1\}$ and $J=\{2\}$. Thus $\mathcal{U}=\mathcal{D}$.
 Theorem \ref{T1} indicates that  $\mathbb{F}_{q}[2V]^{\OO}=\mathbb{F}_{q}[x_{1},x_{2},y_{1},y_{2}]^{\OO}$ is generated by $q+3$ invariants: $N_1=x_1y_1,N_2=x_2y_2,U_{12}=x_1y_2+x_2y_1$, and $B_{k}=x_1^kx_2^{q-1-k}+y_1^ky_2^{q-1-k}$ for $0\leqslant k\leqslant q-1$; See Section \ref{sec7} for more examples.
\hbo} \end{exam}

It follows from \cite[Proposition 16]{Kem96}  that 
$\mathbb{F}_{q}[2V]^{\OO\times\OO}$ is a polynomial algebra generated by $\{N_{1},N_{2},B_{0},B_{q-1}\}$.
Moreover,  by  \cite[Lemma 2.6.3]{CW11}, we see that $\{N_{1},N_{2},B_{0},B_{q-1}\}$ is a homogeneous system of parameters for $\mathbb{F}_{q}[2V]^{\OO}$. Note that the cyclic group  $C_{2}$ of order 2 is the Sylow $2$-subgroup of $\OO$ and $\mathbb{F}_{q}[2V]^{C_{2}}$ is a hypersurface (so Cohen-Macaulay) algebra; see \cite[Section 1.2]{CW11}). Thus it follows from \cite[Theorem 1]{CHP91}  that $\mathbb{F}_{q}[2V]^{\OO}$ is Cohen-Macaulay.

In this paper, we also construct a free basis for $\mathbb{F}_{q}[2V]^{\OO}$ over  $\mathbb{F}_{q}[2V]^{\OO\times\OO}$ by showing the following second result.

 \begin{thm}\label{T2} 
 The set
$$\Big\{U_{12}^{i}\mid 0\leqslant i\leqslant \frac{q}{2}\Big\}\cup \Big\{B_{k}\mid 1\leqslant k\leqslant q-2\Big\}\cup\Big\{B_{i}B_{j}\mid 1\leqslant i,j\leqslant q-2\textrm{ and }i+j=q-1\Big\}$$ is a basis for $\mathbb{F}_{q}[2V]^{\OO}$ as a free $\mathbb{F}_{q}[2V]^{\OO\times\OO}$-module.
 \end{thm}

The Hilbert ideal $\mathfrak{h}_{W}(G)$ associated with an invariant ring  $\FF[W]^{G}$ is the ideal in $\FF[W]$ generated by 
all invariants of positive degree, namely, $\mathfrak{h}_{W}(G)=(\FF[W]^{G}_{+})\FF[W]$.
Derksen and Kemper  have conjectured  that  
 $\mathfrak{h}_{W}(G)$ can be generated by invariants of degree $\leqslant |G|$ for any finite group $G$ and any representation $W$; see \cite[Conjecture 3.8.6 (b)]{DK02}.

The third purpose of this paper is to find a generating set of $\h$. The following Theorem \ref{T3} shows 
that $\h$ can be generated by invariants of degree $\leqslant q-1=\frac{|\OO|}{2}$, confirming the conjecture of
Derksen and Kemper  in this particular case. 

 \begin{thm}\label{T3}
The Hilbert ideal $\h$ can be generated by $\mathcal{N}\cup\mathcal{U}\cup\mathcal{B}$.
 \end{thm}

This paper is organized as follows:  Section \ref{sec2} contains preliminaries and some basic constructions.  Our main lemmas, which explain the main idea in the proof of Theorem \ref{T1}, are contained in Section \ref{sec3}. Section \ref{sec4}, together with several technical lemmas in Section \ref{sec5}, gives a complete proof of Theorem \ref{T1}.
Section \ref{sec6}  is devoted to giving a proof of Theorem \ref{T2}. In Section \ref{sec7}, we provide more examples to 
illustrate how large the number of generators in Theorem \ref{T1} is; a proof of Corollary \ref{noethernumber} is also given. Section \ref{sec8}  contains a proof of Theorem \ref{T3}.
In Section \ref{sec9}, we discuss  the orthogonal group of minus type $\Oo$, and its invariant ring $\Ir$.

We close this introductory section with some notation and conventions. 

\begin{nc}\label{nc}
{\rm We always assume that  $\Fq$ is a finite field of characteristic $2$. We define $\overline{m}:=\big\{1,2,\dots,m\big\}$.
The Greek letters, $\upalpha,\bet,\dots$, denote vectors in $\N^{m}$.
For any vector $$\upalpha=\va\in\N^{m},$$ we define $|\upalpha|:=\sum_{i=1}^{m}\upalpha_{i}$  and 
\begin{equation}
\label{ }
B_{\upalpha}  :=  \xa+\ya.
\end{equation}

Let $I\subseteq \overline{m}$ be a nonempty subset and $\upalpha=\va\in\N^{m}$ be any vector with $\upalpha_{i}=0$ for all $i\notin I$. We define $x^{\upalpha}_{I} := \prod_{i\in I} x_{i}^{\upalpha_{i}}\in\Fq[mV]$. Similarly, we define 
$y^{\bet}_{J} :=  \prod_{j\in J} y_{j}^{\bet_{j}} \in\Fq[mV]$ for a nonempty subset $J\subseteq \overline{m}$ and any vector $\bet$ with $\bet_{j}=0$ for all $j\notin J$. We also define
 \begin{eqnarray}
d_{I,J}(\upalpha,\bet)&:=&x^{\upalpha}_{I}\cdot y^{\bet}_{J}+y^{\upalpha}_{I}\cdot x^{\bet}_{J}\\
d_{I,J}&:=&d_{I,J}(\overline{1},\overline{1})=x_{I}\cdot y_{J}+y_{I}\cdot x_{J}
\end{eqnarray} 
where $\overline{1}$ is the vector whose the $i$-th component is 1 for every $i\in I$ (or $J$) and other components are zero.
Given  two nonempty subsets $I,J\subseteq \overline{m}$, we say that \textit{$I$ is less than $J$}, denoted $I<J$, if $i<j$ for all $i\in I$ and all $j\in J$.
\hbo}\end{nc}

\section{Preliminaries} \label{sec2}
\setcounter{equation}{0}
\renewcommand{\theequation}
{2.\arabic{equation}}
\setcounter{theorem}{0}
\renewcommand{\thetheorem}
{2.\arabic{theorem}}

\noindent Let $\Fq$ denote a finite field of characteristic $2$. Recall that
a square matrix $A=(a_{ij})$ over any field $k$ is said to be \textit{alternate} if $a_{ij}=-a_{ji}$ and $a_{ii}=0$. 
Thus a square matrix over $\Fq$ is alternate if and only if it is symmetric with diagonals zero. Suppose  $A$ and $B$ are two $n\times n$ matrices over $\Fq$. We say that $A$ is \textit{congruent} to $B$, denoted  $A\equiv B$,  if $A-B$ is an alternate matrix.
We choose a fixed element $w\notin\{x^{2}+x\mid x\in \Fq\}$. 
It is well-known that the two-dimensional \textit{orthogonal groups}, up to isomorphism, are just the following two types:
    \begin{eqnarray}
\OO&=&\Big\{T\in \GL\mid T\cdot O^{+}\cdot T'\equiv O^{+}\Big\}\\
\Oo&=&\Big\{T\in \GL\mid T\cdot O^{-}\cdot T'\equiv O^{-}\Big\}
\end{eqnarray}
  where 
 $O^{+}=\begin{pmatrix}
      0 & 1 \\
     0  & 0\\
    \end{pmatrix}$ and $O^{-}=\begin{pmatrix}
      w&1\\
       0&w
    \end{pmatrix}
  $; see for example, \cite[page 188]{TW06} or \cite{Wan93}. Note that 
  $|\OO|=2(q-1)$ and  $|\Oo|=2(q+1).$

 \begin{rem}{\rm
 The  $2$-dimensional orthogonal groups over a finite field $\mathbb{F}_{q}$ of characteristic $p>2$ have also two isomorphism classes: $\OO$ and $\Oo$, with the orders $2(q-1)$ and $2(q+1)$ respectively. Since $p$ does not divide $2(q-1)$ and $2(q+1)$,  the invariants for $\OO$ and $\Oo$ with the standard representations are nonmodular. In this case many classical tools, such as Molien's formula and Noether's bound theorem, can be applied. Thus we ignore this case and emphasize the modular case: char$(\Fq)=2$.
We also refer to  \cite[page 213]{NS02}, which discusses the generator problem for the invariant ring of $O_{2}^{\pm}(\FF_{p})$ in the nonmodular case.
 \hbo }\end{rem}  

From now on we always assume that char$(\Fq)=2$ and $q=2^{s}$ with $s\geqslant 2$. The orthogonal group of plus type $\OO$ is generated by 
\begin{equation}
\label{ }
\sig:=\begin{pmatrix}
      0 & 1 \\
     1  & 0\\
    \end{pmatrix}\textrm{ and } \uptau_a:=\begin{pmatrix}
      a & 0 \\
     0  & a^{-1}\\
    \end{pmatrix},
\end{equation}
where $a\in \Fq^{\times}.$ Let $V$ denote the 2-dimensional standard representation of 
    $\OO$ over $\Fq$ and $\OO$ act on $mV$ diagonally.
The action of $\OO$ on $\Fq[mV]:=\FqmV$ is given by
  \begin{eqnarray}
\sig(x_i)=y_i,&&\sig(y_i)=x_i\\
 \uptau_a(x_i)=a^{-1}\cdot x_i,& &  \uptau_a(y_i)=a\cdot y_i
\end{eqnarray}
for  $1\leqslant i\leqslant m$.

\begin{prop}   
$B_{\upalpha}\in\IR$ if and only if $q-1$ divides $|\upalpha|$.
\end{prop}

\begin{proof}   
If $q-1$ divides $|\upalpha|$, then a direct calculation shows that $\sig(B_{\upalpha})=B_{\upalpha}=\uptau_{a}(B_{\upalpha})$. Thus $B_{\upalpha}$ is an  $\OO$-invariant. Conversely, since $\uptau_{a}(B_{\upalpha})=a^{-|\upalpha|}\cdot x_{I}^{\upalpha}+a^{|\upalpha|}\cdot y_{I}^{\upalpha}=x_{I}^{\upalpha}+y_{I}^{\upalpha}$, it follows that $a^{-|\upalpha|}-1=0=a^{|\upalpha|}-1.$ Hence, $q-1$ divides $|\upalpha|$.
\end{proof}

\begin{prop}   \label{2.3}
The invariant ring
$\Fq[V]^{\OO}=\Fq[x,y]^{\OO}=\Fq[xy,x^{q-1}+y^{q-1}]$ is a polynomial algebra. 
\end{prop}

\begin{proof}   
It follows immediately from  \cite[Proposition 16]{Kem96}.
\end{proof}

\section{The Main Lemma} \label{sec3}
\setcounter{equation}{0}
\renewcommand{\theequation}
{3.\arabic{equation}}
\setcounter{theorem}{0}
\renewcommand{\thetheorem}
{3.\arabic{theorem}}

\noindent The following criterion will be very useful for our proof of Theorem \ref{T1}. 

\begin{lem}\label{Chen2017}
Let $\FF$ be any field and  $W$ be an $n$-dimensional faithful representation of a finite group $G$
over $\FF$. Let $H\leqslant G$ be a proper subgroup with $[G:H]^{-1}\in \FF$. Suppose
$\{f_{1},f_{2},\dots,f_{m}\}\subset \FF[W]^{G}_{+}\subset \FF[W]^{H}$ is a set of homogeneous polynomials. 
Let $A=\FF[f_{1},f_{2},\dots,f_{m}]$ and $\J$ denote the ideal generated by $\{f_{1},f_{2},\dots,f_{m}\}$ in $\FF[W]^{H}$. We denote the Reynolds operator by:
\begin{equation}
\label{ }
\R=\frac{1}{[G:H]} {\rm Tr}_{H}^{G}:\FF[W]^{H}\longrightarrow \FF[W]^{G}.
\end{equation}
Suppose  that $\Delta\cup\{1\}$ is a homogenous generating set of $\FF[W]^{H}$ as an $A$-module and $\del\notin A$ for any $\del\in\Delta$, i.e., 
$$\FF[W]^{H}=A+\sum_{\del\in\Delta} \del \cdot A.$$
 If $\R(\del)\in\J$ for all $\del\in\Delta$, then
$\FF[W]^{G}=A$.
\end{lem}

\begin{proof}
Since $A\subseteq\FF[W]^{G}$, it suffices to show that $\FF[W]^{G}\subseteq A$. Note that $\R$ is surjective and degree-preserving. Thus
we only need to show the following \textit{claim}: $$\R(g)\in A,\textrm{ for any homogeneous element }g\in \FF[W]^{H}.$$
We use the induction on degree of $g$ to prove this claim. Suppose $g\in \FF[W]^{H}$ is any homogeneous element and suppose $\R(h)\in A$ for any homogeneous element $h\in \FF[W]^{H}$ with $\deg(h)<\deg(g)$.  Since
$\Delta\cup\{1\}$ is a homogeneous generating set of $\FF[W]^{H}$ as an $A$-module, we may write $$g=a_{0}+\sum_{i=1}^{r} a_{i}\cdot \del_{i}$$
where all $\del_{i}\in\Delta$, $a_{i}\in A$ and $r\in\N^{+}$. 
Since every $\R(\del_{i})\in\J$, we may write
$$\R(\del_{i})=g_{i1}f_{1}+g_{i2}f_{2}+\dots+g_{im}f_{m}$$
where $g_{ij}\in \FF[W]^{H}$ are homogeneous. Since $\R$ is an $\FF[W]^{G}$-module homomorphism, we have
$$\R(g)=\R^{2}(g)=\R\left(a_{0}+\sum_{i=1}^{r} a_{i}\cdot \R(\del_{i})\right)=a_{0}+\sum_{i=1}^{r} a_{i}\sum_{j=1}^{m}f_{j}\cdot \R(g_{ij}).$$
Note that $\deg(f_{j})>0$, so $\deg(g_{ij})<\deg(g)$
for all $1\leqslant i\leqslant r$ and $1\leqslant j\leqslant m$. Thus $\R(g_{ij})\in A$. Therefore, $\R(g)\in A$ and the claim holds.
\end{proof}

This lemma leads us to reduce the calculation of $\IR$ to computing $\Fq[mV]^{P}$, where $P$
denotes the Sylow 2-subgroup of $\OO$. On the other hand, 
we see that $P\cong C_2$ is the cyclic group of order 2. It is well-known that any 2-dimensional indecomposable modular representation of the cyclic group $C_p=\langle \sig\rangle$ of order $p$ is  equivalent to  the representation defined by 
  $$\sig\mapsto\begin{pmatrix}
    1  &  1  \\
     0 &  1
\end{pmatrix},$$
see for example,  \cite[page 105]{CW11}.
Since the invariant rings for equivalent representations are isomorphic,  we derive  the following result immediately from Richman's Theorem, see \cite{Ric90} or  \cite{CSW10}.
  
\begin{thm} \label{Ric1990} 
Let $\Fq$ be a finite field of characteristic 2 and $P=\langle \sig\rangle$ be the cyclic group of order 2. 
Suppose that $\Fq[mV]=\FqmV$ is a polynomial algebra on which $P$ acts by permutation, i.e.,  $\sig(x_i)=y_i$ and $\sig(y_i)=x_i$ for all $1\leqslant i\leqslant m$. Then  $\Fq[mV]^{P}$ is generated by 
\begin{eqnarray*}
\mathcal{L} & = & \Big\{L_i=x_i+y_i\mid 1\leqslant i\leqslant m\Big\} \\
\mathcal{N} & = & \Big\{N_i=x_iy_i\mid 1\leqslant i\leqslant m\Big\}\\
\mathcal{U} &=& \Big\{U_{ij}=x_iy_j+x_jy_i\mid 1\leqslant i<j\leqslant m\Big\}\\
\mathcal{B}'&=& \Big\{B_\upalpha=\xa+\ya\mid 0\leqslant\upalpha_{1},\dots,\upalpha_{m}\leqslant 1\Big\}.
\end{eqnarray*}
\end{thm}

  \begin{rem}\label{r3.3}
  {\rm
  Note that in Theorem \ref{Ric1990}, the set $\mathcal{L}$ is contained in $\mathcal{B}'$. Moreover, $\Fq[mV]^{P}$ is generated minimally by $\mathcal{L}\cup \mathcal{N} \cup\mathcal{U} \cup \mathcal{B}''$, where $$\mathcal{B}'':=\Big\{B_{\upalpha}\in \mathcal{B}'~\big|~ |\upalpha|\geqslant 3\Big\}$$
see \cite[Corollary 4.4]{SW02}.
\hbo}\end{rem}

\section{Proof of  Theorem \ref{T1}}\label{sec4}
\setcounter{equation}{0}
\renewcommand{\theequation}
{4.\arabic{equation}}
\setcounter{theorem}{0}
\renewcommand{\thetheorem}
{4.\arabic{theorem}}

\noindent We begin this section with the following well-known result whose proof could be found in  \cite[Lemma 9.0.2]{CW11}.

\begin{lem}\label{4.1} 
Let $q=p^{s}$ be a prime power and $e\in\N^{+}$. Then
$$\sum_{a\in \mathbb{F}_{q}^{\times}}a^{e}=\begin{cases}
     -1, & \text{if } q-1\text{ divides } e, \\
      0, & \text{otherwise}.
\end{cases}$$
\end{lem}

We define $L^{e}:= L_{1}^{e_{1}}L_{2}^{e_{2}}\cdots L_{m}^{e_{m}}$ and $N^{\del}:=N_{1}^{\del_{1}}N_{2}^{\del_{2}}\cdots N_{m}^{\del_{m}}$ for any vectors $e=(e_{1},e_{2},\dots,e_{m})\in \N^{m}$ and $\del=(\del_{1},\del_{2},\dots,\del_{m})\in \N^{m}$. The following result is an immediate consequence from  \cite[Proposition 3.4]{CW14}.

\begin{lem}\label{CW2014} 
For any $B_{\upalpha},B_{\bet}\in \mathcal{B}'$,  we have
\begin{equation}
\label{ }
B_{\upalpha}\cdot B_{\bet}
= \sum L^{e}\cdot N^{\del} \cdot B_{\gam}+N^{\del'}\cdot\sum L^{e'} \cdot B_{\gam'}
\end{equation}
where two sums are both finite, the vectors $e,e',\del,\del'\in \N^{m}$, and  $B_{\gam},B_{\gam'}\in \mathcal{B}'$.
\end{lem}

 \begin{proof}[Proof of the first assertion of Theorem \ref{T1}]
We define $\mathcal{S} :=  \mathcal{N}\cup \mathcal{B}\cup \mathcal{D}$, which will be our desired generating set as $\{f_{1},f_{2},\dots,f_{m}\}$ in Lemma \ref{Chen2017}  and let $\J$ denote the ideal generated by $\mathcal{S}$ in $\Fq[mV]^{P}$. Then the Reynolds operator 
  \begin{equation}
\label{ }
\R:=\R_{P}^{\OO}: \Fq[mV]^{P}\longrightarrow \Fq[mV]^{\OO},\quad f\mapsto \frac{1}{[\OO:P]}\sum_{a\in \Fq^{\times}}\uptau_{a}\cdot f= \sum_{a\in \Fq^{\times}}\uptau_{a}\cdot f
\end{equation}
 is a surjective homomorphism of $\Fq[mV]^{\OO}$-modules.
 
By Lemma \ref{Chen2017} and Theorem \ref{Ric1990}, it suffices to show 
 that the image of any non-constant polynomials in $\Fq[mV]^{P}$ with following form
 \begin{equation}
\label{ }
\left(\prod_{i=1}^{m} N_{i}^{\upalpha_{i}}\right)\left(\prod_{1\leqslant i<j\leqslant m} U_{ij}^{\bet_{ij}}\right)\left(\prod_{B_{\gam}\in \mathcal{B}'}B_{\gam}^{e_{\gam}}\right)
\end{equation}
under $\R$ belongs to $\J$, where $\upalpha_{i},\bet_{ij},e_{\gam}\in \N$.  Note that any  $\Fq[\mathcal{S}]$-module generating set $\Delta$ of $\Fq[mV]^{P}$  consists of elements of the above forms, which means that here we actually give a 
proof for a general result so that the conditions in Lemma \ref{Chen2017} are satisfied. 

Since all $N_{i}, U_{ij}\in \J$ and $\R$ preserves all $\OO$-invariants,  it is sufficient to prove that
\begin{equation}
\label{ }
\R\left(\prod_{B_{\gam}\in \mathcal{B}'}B_{\gam}^{e_{\gam}}\right)\in \J,
\end{equation}
where $\deg\left(\prod\limits_{B_{\gam}\in \mathcal{B}'}B_{\gam}^{e_{\gam}}\right)>0.$ 
By Lemma \ref{CW2014}, it suffices to prove the following three cases:
\begin{eqnarray}
\R(B_{\upalpha})&\in&\J, \\
\R(L^{\upalpha})&\in&\J,\\
\R(L^{\upalpha}\cdot B_{\bet})&\in&\J,
\end{eqnarray}
where $B_{\upalpha},B_{\bet}\in \mathcal{B}'$ and $L^{\upalpha}=L_{1}^{\upalpha_{1}}L_{2}^{\upalpha_{2}}\cdots L_{m}^{\upalpha_{m}}$
are polynomials with positive degree. Our proof  will be completed by applying 
the following Lemmas \ref{5.4},  \ref{5.5} and  \ref{5.6} respectively.
 \end{proof}

 \begin{proof}[Proof of the second assertion of Theorem \ref{T1}]
It is sufficient to show that every element in $\mathcal{N} \cup \mathcal{B} \cup \mathcal{D}$ is indecomposable. 
The fact that $\IR\subseteq\Fq[mV]^{P}$, together with that all $N_{i}$ and $U_{ij}$ are indecomposable in $\Fq[mV]^{P}$ (Remark \ref{r3.3}) implies that all $N_{i}$ and $U_{ij}$ are indecomposable in $\IR$. Note that the elements in $\mathcal{D}$ can be separated into two classes: $$\mathcal{D}_{1}=\{d_{I,J}:|J|=|I|\}\textrm{ and }\mathcal{D}_{2}=\{d_{I,J}:|J|=|I|+q-1\}.$$

For any $B_{\upalpha}\in \mathcal{B}$, assume by way of contradiction  that $B_{\upalpha}$ is decomposable.  Since $|B_{\upalpha}|=q-1$ and every element in $\mathcal{D}_{2}$ has degree $>q-1$, it does not factor using elements from $\mathcal{D}_{2}$. Note that all elements in $\mathcal{N}\cup \mathcal{D}_{1}$ have even degree, so any product of them has even degree. 
However, $|B_{\upalpha}|=q-1$ is odd, 
thus $B_{\upalpha}$  does not factor using  elements from $\mathcal{N}\cup \mathcal{D}_{1}$.
Thus $B_{\upalpha}$ factor using only elements from $\mathcal{B}-\{B_{\upalpha}\}$.
Since any element in $\mathcal{B}$ has the same degree, $B_{\upalpha}$ is a linear combination among $\mathcal{B}-\{B_{\upalpha}\}$ over $\Fq$. This contradiction shows that $B_{\upalpha}$ is indecomposable.

By Shank and Wehlau \cite[Corollary 4.4]{SW02}, we have seen that 
$\prod_{i\in I}x_{i}+\prod_{i\in I}y_{i}$ is indecomposable in $\Fq[mV]^{P}$ for any $I\subseteq \overline{m}$
with $|I|>2$. Thus choosing a suitable basis for $mV$, we also deduce that 
$$d_{I,J}=x_{I}y_{J}+y_{I}x_{J}=\prod_{i\in I}x_{i}\cdot \prod_{j\in J}y_{j}+\prod_{i\in I}y_{i}\cdot \prod_{j\in J}x_{j}$$
is indecomposable in $\Fq[mV]^{P}$ for any $\emptyset\neq I<J\subseteq \overline{m}$
with $|I|+|J|>2$. Thus for any $d_{I,J}\in \mathcal{D}$ with $|I|+|J|>2$, it is indecomposable in $\IR$. 
Since any element in $\mathcal{D}$ has degree $\geqslant 2$, we need only to show that the 
elements in $\mathcal{D}$ with degree 2 are indecomposable. In fact this set of elements of degree 2 just coincides with $\mathcal{U}$.
We have seen that every $U_{ij}$ is indecomposable. This completes the proof.
 \end{proof}

\section{Lemmas}\label{sec5}
\setcounter{equation}{0}
\renewcommand{\theequation}
{5.\arabic{equation}}
\setcounter{theorem}{0}
\renewcommand{\thetheorem}
{5.\arabic{theorem}}

\noindent We follow the notations in previous section and begin with a simple but useful observation:

   \begin{lem}\label{5.1}
$\R(\J)\subseteq \J.$
  \end{lem}
  
\begin{proof}
For any $f\in \J$, we may write $f=\sum a_{i}\cdot f_{i}$ with $a_{i}\in \mathcal{S}$ and $f_{i}\in  \Fq[mV]^{P}$. Since $\R$ is an $\IR$-module homomorphism, we have
 $\R(f)=\R(\sum a_{i}\cdot f_{i})=\sum a_{i}\cdot \R(f_{i})\in \J$. Thus $\R(\J)\subseteq \J.$
\end{proof}

  \begin{lem}\label{5.2}
  Let $\upalpha\in\N^{m}$ be any vector with $|\upalpha|>0$ and $B_{\upalpha}=\xa+\ya\in \Fq[mV]^{P}$. Then for any $e\in \mathbb{N}^{+}$, we have $B_{\upalpha}^{e}\equiv (\xa)^{e}+(\ya)^{e}$ $(\mo\J)$.
In particular, $L_{i}^{e}\equiv x_{i}^{e}+y_{i}^{e}$ $(\mo\J)$,  for all $i\in \overline{m}$.
  \end{lem}

  \begin{proof}
Since $|\upalpha|>0$, there exists some $i\in \overline{m}$
  such that $\upalpha_{i}>0$. Without loss of generality, we suppose $\upalpha_{1}> 0$. Define $x^{\upalpha}=\xa$ and $y^{\upalpha}=\ya$.
  By the binomial formula, we have
  $$B_{\upalpha}^{e}=(x^{\upalpha}+y^{\upalpha})^{e} = (x^{\upalpha})^{e}+\bigg[\sum_{k=1}^{e-1} {e\choose e-k}(x^{\upalpha})^{e-k}(y^{\upalpha})^{k}\bigg]+(y^{\upalpha})^{e}.$$
We \textit{claim} that  $\sum_{k=1}^{e-1} {e\choose e-k}(x^{\upalpha})^{e-k}(y^{\upalpha})^{k}\in \J$.
Define $b_{k}:={e\choose e-k}(x^{\upalpha})^{e-k}(y^{\upalpha})^{k}+{e\choose k}(x^{\upalpha})^{k}(y^{\upalpha})^{e-k}$ for $1\leqslant k\leqslant\frac{e-1}{2}$ (when $e$ is odd) and $1\leqslant k\leqslant\frac{e}{2}-1$ (when $e$ is even). 
Further, when $e$ is even, since ${e\choose e/2}$ is even, we have
${e\choose e/2}(x^{\upalpha})^{e/2}(y^{\upalpha})^{e/2}=0$.
Thus it is sufficient to show that every $b_{k}\in\J$.
Since char$(\Fq)=2$ and ${e\choose e-k}={e\choose k}$,  $b_{k}=0$ whenever ${e\choose k}$ is even.
Suppose ${e\choose k}$ is odd in $b_{k}$. Since char$(\Fq)=2$, then ${e\choose k}=1$ in $\Fq$.
Note that $\upalpha_{1}\geqslant 1$ and $\frac{e-1}{2}\geqslant k\geqslant 1$,  we have
  \begin{eqnarray*}
b_{k}&=&(x^{\upalpha})^{e-k}(y^{\upalpha})^{k}+(x^{\upalpha})^{k}(y^{\upalpha})^{e-k} \\
 & = & (x_1^{\upalpha_1(e-k)}x_2^{\upalpha_2(e-k)}\cdots x_m^{\upalpha_m(e-k)}) (y_1^{\upalpha_1k}y_2^{\upalpha_2k}\cdots y_m^{\upalpha_mk}) +\\
 && (y_1^{\upalpha_1(e-k)}y_2^{\upalpha_2(e-k)}\cdots y_m^{\upalpha_m(e-k)}) (x_1^{\upalpha_1k}x_2^{\upalpha_2k}\cdots x_m^{\upalpha_mk}) \\
 &=& N_{1}\cdot\bigg[(x_1^{\upalpha_1(e-k)-1}x_2^{\upalpha_2(e-k)}\cdots x_m^{\upalpha_m(e-k)}) (y_1^{\upalpha_1k-1}y_2^{\upalpha_2k}\cdots y_m^{\upalpha_mk}) +\\
 && (y_1^{\upalpha_1(e-k)-1}y_2^{\upalpha_2(e-k)}\cdots y_m^{\upalpha_m(e-k)}) (x_1^{\upalpha_1k-1}x_2^{\upalpha_2k}\cdots x_m^{\upalpha_mk})\bigg] \in \J. 
\end{eqnarray*}
Thus the claim follows and $B_{\upalpha}^{e}\equiv(x^{\upalpha})^{e}+(y^{\upalpha})^{e}$ $(\mo\J)$.
In particular, when $B_{\upalpha}=L_{i}=x_{i}+y_{i}$, we have $L_{i}^{e}\equiv x_{i}^{e}+y_{i}^{e}$ $(\mo\J)$.
  \end{proof}

 \begin{lem}\label{5.3}
For any nonempty sets $I,J\subseteq \overline{m}$ and  $d_{I,J}(\upalpha,\bet)=x^{\upalpha}_{I}\cdot y^{\bet}_{J}+y^{\upalpha}_{I}\cdot x^{\bet}_{J}$ with all exponents $\upalpha_{i},\bet_{j}\geqslant 1$, we have $\R(d_{I,J}(\upalpha,\bet))\in \J.$
  \end{lem}
 
 \begin{proof} 
Note that $d_{I,J}(\upalpha,\bet)$ is a $P$-invariant. 
The proof will be separated into two cases: $I\cap J\neq \emptyset$ and $I\cap J=\emptyset$. 
For the first case, we suppose that there exists an integer $k\in I\cap J$. Since all  $\upalpha_{i},\bet_{j}\geqslant 1$, we have
  \begin{eqnarray*}
d_{I,J}(\upalpha,\bet)&=&(x_{k}^{\upalpha_{k}}y_{k}^{\bet_{k}})\cdot x^{\upalpha}_{I-\{k\}}\cdot y^{\bet}_{J-\{k\}}+(y_{k}^{\upalpha_{k}}x_{k}^{\bet_{k}})\cdot y^{\upalpha}_{I-\{k\}}\cdot x^{\bet}_{J-\{k\}} \\
 & = &  N_{k}\cdot\Big[(x_{k}^{\upalpha_{k}-1}y_{k}^{\bet_{k}-1})\cdot x^{\upalpha}_{I-\{k\}}\cdot y^{\bet}_{J-\{k\}}+(y_{k}^{\upalpha_{k}-1}x_{k}^{\bet_{k}-1})\cdot y^{\upalpha}_{I-\{k\}}\cdot x^{\bet}_{J-\{k\}}\Big]\in \J.
\end{eqnarray*}
By Lemma \ref{5.1},
 $\R(d_{I,J}(\upalpha,\bet))\in \J$ in this case. 
 
Secondly, we suppose   $I\cap J=\emptyset$. This situation can be separated into two subcases: 

\textsc{Subcase 1.}  For all $i\in I$ and all $j\in J$, $\upalpha_{i}=1=\bet_{j}$. For any $i\in I$, if there exists an integer $j\in J$
 such that $i>j$, then 
 \begin{eqnarray*}
d_{I,J} (\upalpha,\bet) &=& d_{I,J} \\
& = & (x_{i}y_{j})(x_{I-\{i\}}\cdot y_{J-\{j\}})+ (y_{i}x_{j})(y_{I-\{i\}}\cdot x_{J-\{j\}})\\
 & = & (U_{ji}+y_{i}x_{j})(x_{I-\{i\}}\cdot y_{J-\{j\}})+ (U_{ji}+x_{i}y_{j})(y_{I-\{i\}}\cdot x_{J-\{j\}})\\ 
 &=& U_{ji}\cdot(x_{I-\{i\}}\cdot y_{J-\{j\}}+y_{I-\{i\}}\cdot x_{J-\{j\}})+\\
 && \Big[x_{(I-\{i\})\cup\{j\}}\cdot y_{\{i\}\cup(J-\{j\})}+y_{(I-\{i\})\cup\{j\}}\cdot x_{\{i\}\cup(J-\{j\})}\Big].
\end{eqnarray*}
Since $U_{ji}\cdot(x_{I-\{i\}}\cdot y_{J-\{j\}}+y_{I-\{i\}}\cdot x_{J-\{j\}})\in \J$ and $\R(\J)\subseteq \J$, if we want to prove $\R(d_{I,J})\in \J$, it is sufficient to show that 
$$\R\Big[x_{(I-\{i\})\cup\{j\}}\cdot y_{\{i\}\cup(J-\{j\})}+y_{(I-\{i\})\cup\{j\}}\cdot x_{\{i\}\cup(J-\{j\})}\Big]\in \J.$$
Proceeding in this way, we  need to show that 
\begin{equation}
\label{ }
\R(x_{I}y_{J}+y_{I}x_{J})\in \J,
\end{equation}
where $I<J.$ On the other hand, whenever $I<J$,
\begin{eqnarray*}
\R(x_{I}y_{J}+y_{I}x_{J})& = & \sum_{a\in \mathbb{F}_{q}^{\times}} a^{|J|-|I|}x_{I}y_{J}+\sum_{a\in \mathbb{F}_{q}^{\times}} a^{|I|-|J|}y_{I}x_{J}\\
 &=&(\sum_{a\in \mathbb{F}_{q}^{\times}} a^{|J|-|I|})\cdot (x_{I}y_{J}+y_{I}x_{J})\\
 &=& \begin{cases}
    x_{I}y_{J}+y_{I}x_{J}, & \text{if } q-1\text{ divides } |J|-|I|, \\
      0, & \text{otherwise}.
\end{cases}
\end{eqnarray*}
The last equation follows from Lemma \ref{4.1}.
We have to show that $d_{I,J}\in \J$ if $q-1$ divides $|J|-|I|$. By the symmetry of $d_{I,J}$, we may write $|J|-|I|=(q-1)\cdot r$, where $r\in \mathbb{N}$. We use induction on $r$.
 If $r=0$ or 1, we are done. Let $I'\subseteq J$ denote the  subset such that $|I'|=|I|$ and $J-I'< I'$.  For any $k=1,2,\dots,r$,  we let $J_{k}\subseteq J-I'$ denote the subsets such that $|J_{k}|=q-1$ and
$J_{1}<J_{2}<\dots<J_{r}$.
Then 
\begin{eqnarray*}
d_{I,J}& = & x_{I}y_{J-I'}y_{I'}+y_{I}x_{J-I'}x_{I'} \\
 & = &  x_{I}y_{J_{1}}y_{J_{2}}\cdots y_{J_{r}}y_{I'}+y_{I}x_{J_{1}}x_{J_{2}}\cdots x_{J_{r}}x_{I'}\\
 &=& x_{I}(d_{J_{1}}+x_{J_{1}})y_{J_{2}}\cdots y_{J_{r}}y_{I'}+ y_{I}(d_{J_{1}}+y_{J_{1}})x_{J_{2}}\cdots x_{J_{r}}x_{I'}\\
 &=& d_{J_{1}}(x_{I}y_{J_{2}}\cdots y_{J_{r}}y_{I'}+y_{I}x_{J_{2}}\cdots x_{J_{r}}x_{I'})+(x_{I}x_{J_{1}}y_{J_{2}}\cdots y_{J_{r}}y_{I'}+y_{I}y_{J_{1}}x_{J_{2}}\cdots x_{J_{r}}x_{I'}),
\end{eqnarray*}
where $d_{J_{1}}:=\prod_{j\in J_{1}} x_{j}+\prod_{j\in J_{1}} y_{j}\in \mathcal{B}$ because $|J_{1}|=q-1$.
To see that $d_{I,J}\in \J$, it suffices to show that 
$x_{I\cup J_{1}}y_{J_{2}}\cdots y_{J_{r-1}} y_{J_{r}\cup I'}+y_{I\cup J_{1}}x_{J_{2}}\cdots x_{J_{r-1}} x_{J_{r}\cup I'}\in \J,$
which actually follows from the induction hypothesis.
This finishes  the proof for the first subcase.

\textsc{Subcase 2.} For some $i\in I$ (resp. $j\in J$), we have $\upalpha_{i}\geqslant 2$ (resp. $\bet_{j}\geqslant 2$). By the  symmetry of   $d_{I,J}(\upalpha,\bet)$, we suppose that
there exists an $i\in I$ such that $\upalpha_{i}\geqslant 2$. 
For any $j\in J$, we have
\begin{eqnarray*}
d_{I,J}(\upalpha,\bet)&=&x^{\upalpha}_{I}\cdot y^{\bet}_{J}+y^{\upalpha}_{I}\cdot x^{\bet}_{J} \\
 & = & (x_{i}y_{j})(x^{\upalpha}_{I-\{i\}}x_{i}^{\upalpha_{i}-1})(y^{\bet}_{J-\{j\}}y_{j}^{\bet_{j}-1})+(y_{i}x_{j})(y^{\upalpha}_{I-\{i\}}y_{i}^{\upalpha_{i}-1})(x^{\bet}_{J-\{j\}}x_{j}^{\bet_{j}-1})\\
  & = & (U_{ij}+y_{i}x_{j})(x^{\upalpha}_{I-\{i\}}x_{i}^{\upalpha_{i}-1})(y^{\bet}_{J-\{j\}}y_{j}^{\bet_{j}-1})+(U_{ij}+x_{i}y_{j})(y^{\upalpha}_{I-\{i\}}y_{i}^{\upalpha_{i}-1})(x^{\bet}_{J-\{j\}}x_{j}^{\bet_{j}-1})\\
  &=& U_{ij}\cdot\Big[(x^{\upalpha}_{I-\{i\}}x_{i}^{\upalpha_{i}-1})(y^{\bet}_{J-\{j\}}y_{j}^{\bet_{j}-1})+(y^{\upalpha}_{I-\{i\}}y_{i}^{\upalpha_{i}-1})(x^{\bet}_{J-\{j\}}x_{j}^{\bet_{j}-1})\Big]+\\
  && N_{i}\cdot\Big[(x^{\upalpha}_{I-\{i\}}x_{i}^{\upalpha_{i}-2}x_{j})(y^{\bet}_{J-\{j\}}y_{j}^{\bet_{j}-1})+(y^{\upalpha}_{I-\{i\}}y_{i}^{\upalpha_{i}-2}y_{j})(x^{\bet}_{J-\{j\}}x_{j}^{\bet_{j}-1})\Big]
  \end{eqnarray*}
  which belongs to $\J$, so
$\R(d_{I,J}(\upalpha,\bet))\in \R(\J)\subseteq \J.$
 \end{proof}

  \begin{lem}\label{5.4}
For any $\upalpha\in \mathbb{N}^{m}$ with $|\upalpha|>0$ and  $B_{\upalpha}=\xa+\ya$, we have $\R(B_\upalpha)\in \J.$
  \end{lem}
 
  \begin{proof} 
   Indeed,
\begin{eqnarray*}
\R(B_\upalpha)& = & \sum_{a\in \mathbb{F}_{q}^{\times}} a^{-|\upalpha|}x_1^{\upalpha_1}x_2^{\upalpha_2}\cdots x_m^{\upalpha_m}+\sum_{a\in \mathbb{F}_{q}^{\times}} a^{|\upalpha|}y_1^{\upalpha_1}y_2^{\upalpha_2}\cdots y_m^{\upalpha_m}\\
 & = & \left(\sum_{a\in \mathbb{F}_{q}^{\times}} (a^{-1})^{|\upalpha|}\right)x_1^{\upalpha_1}x_2^{\upalpha_2}\cdots x_m^{\upalpha_m}+\left(\sum_{a\in \mathbb{F}_{q}^{\times}} a^{|\upalpha|}\right)y_1^{\upalpha_1}y_2^{\upalpha_2}\cdots y_m^{\upalpha_m}\\
 &=&\left(\sum_{a\in \mathbb{F}_{q}^{\times}} a^{|\upalpha|}\right)\cdot B_\upalpha\\
 &=& \begin{cases}
     B_\upalpha, & \text{if } q-1\text{ divides } |\upalpha|, \\
      0, & \text{otherwise}.
\end{cases}
\end{eqnarray*}
The last equation follows from Lemma \ref{4.1}. We have to prove the \textit{claim} that $B_{\upalpha}\in \mathcal{J}$  for all $\upalpha$ with $|\upalpha|=(q-1)\cdot r$, where $r\in \mathbb{N}^{+}$. 
If $r=1$, this claim holds clearly. Now suppose $r>1$.
We may write  $B_{\upalpha}=x^{\upalpha'}x^{\upalpha''}+y^{\upalpha'}y^{\upalpha''}$, where $|\upalpha'|=q-1$ and $|\upalpha''|=(q-1)(r-1)$.
Then 
\begin{eqnarray*}
B_{\upalpha} & = & (B_{\upalpha'}+y^{\upalpha'}) x^{\upalpha''}+ (B_{\upalpha'}+x^{\upalpha'}) y^{\upalpha''}\\
 & = & B_{\upalpha'}\cdot(x^{\upalpha''}+y^{\upalpha''}) + (x^{\upalpha'} y^{\upalpha''}+y^{\upalpha'} x^{\upalpha''}).
\end{eqnarray*}
Note that $B_{\upalpha'}:=x^{\upalpha'}+y^{\upalpha'}\in \mathcal{B}$. To show that $B_{\upalpha}\in \J$, it suffices to show that 
$x^{\upalpha'} y^{\upalpha''}+y^{\upalpha'} x^{\upalpha''}\in \J.$
However, $x^{\upalpha'} y^{\upalpha''}+y^{\upalpha'} x^{\upalpha''}\in\IR$, so 
$x^{\upalpha'} y^{\upalpha''}+y^{\upalpha'} x^{\upalpha''}=\R(x^{\upalpha'} y^{\upalpha''}+y^{\upalpha'} x^{\upalpha''}).$
By Lemma \ref{5.1}, we have 
$\R(x^{\upalpha'} y^{\upalpha''}+y^{\upalpha'} x^{\upalpha''})\in \J$. Thus $B_{\upalpha}\in \J$ and the claim holds. 
\end{proof}

 \begin{lem} \label{5.5}
For any $L^{\upalpha}=L_1^{\upalpha_1}L_2^{\upalpha_2}\cdots L_m^{\upalpha_m}$ with $|\upalpha|>0$, we have 
$\R(L^{\upalpha}) \in \J.$
  \end{lem}

 \begin{proof}
 Let $I=\big\{i\mid \upalpha_{i}\neq0\big\}\subseteq \overline{m}$. Then
 \begin{eqnarray*}
L^{\upalpha}& = & \prod_{i\in I}L_{i}^{\upalpha_{i}} \\
&=& \prod_{i\in I}(x_{i}+y_{i})^{\upalpha_{i}} \\
&\equiv & \prod_{i\in I}(x_{i}^{\upalpha_{i}}+y_{i}^{\upalpha_{i}}) \quad (\mo\J) \quad(\textrm{by Lemma }\ref{5.2})\\
&=&\sum_{K\subseteq I} (x^{\upalpha}_{K}\cdot y^{\upalpha_{c}}_{K^{c}}+y^{\upalpha}_{K}\cdot x^{\upalpha_{c}}_{K^{c}}),
\end{eqnarray*}
 where $K^{c}=I-K$ denotes the complement of $K$ in $I$, and the sum runs over 
 the representatives of the quotient set of the power set of $I$ on the equivalence relation: 
 $K_{1}\sim K_{2}$ if and only if $K_{2}=K_{1}^{c}.$ It follows from Lemmas \ref{5.3} and \ref{5.4}  that the image of
 every $x^{\upalpha}_{K}\cdot y^{\upalpha_{c}}_{K^{c}}+y^{\upalpha}_{K}\cdot x^{\upalpha_{c}}_{K^{c}}$ under $\R$ belongs to $\J$. Hence, $\R(L^{\upalpha})\in \J.$
 \end{proof}

   \begin{lem} \label{5.6}
For any $B_{\bet}$ and $L^{\upalpha}=L_1^{\upalpha_1}L_2^{\upalpha_2}\cdots L_m^{\upalpha_m}$ with $|\upalpha|>0,|\bet|>0$, we have 
$\R(L^{\upalpha}\cdot B_{\bet}) \in \J.$
  \end{lem}

  \begin{proof}  
 As in the proof of Lemma \ref{5.5}, we have 
 $ L^{\upalpha}\equiv \sum_{K\subseteq I} (x^{\upalpha}_{K}\cdot y^{\upalpha_{c}}_{K^{c}}+y^{\upalpha}_{K}\cdot x^{\upalpha_{c}}_{K^{c}})$
$(\mo\J)$.
Thus to show that 
$\R(L^{\upalpha}\cdot B_{\bet}) \in \J$, it is sufficient to show that the image of
every $(x^{\bet}+y^{\bet})(x^{\upalpha}_{K}\cdot y^{\upalpha_{c}}_{K^{c}}+y^{\upalpha}_{K}\cdot x^{\upalpha_{c}}_{K^{c}})$
 belongs to $\J$. On the other hand,
 $$(x^{\bet}+y^{\bet})(x^{\upalpha}_{K}\cdot y^{\upalpha_{c}}_{K^{c}}+y^{\upalpha}_{K}\cdot x^{\upalpha_{c}}_{K^{c}})=
 (x^{\bet}x^{\upalpha}_{K}\cdot y^{\upalpha_{c}}_{K^{c}}+y^{\bet}y^{\upalpha}_{K}\cdot x^{\upalpha_{c}}_{K^{c}})+(x^{\upalpha}_{K}\cdot y^{\upalpha_{c}}_{K^{c}}y^{\bet}+y^{\upalpha}_{K}\cdot x^{\upalpha_{c}}_{K^{c}}x^{\bet}).$$
 Applying Lemmas \ref{5.3} and \ref{5.4} we see that 
$\R(x^{\bet}x^{\upalpha}_{K}\cdot y^{\upalpha_{c}}_{K^{c}}+y^{\bet}y^{\upalpha}_{K}\cdot x^{\upalpha_{c}}_{K^{c}})$ and
$\R(x^{\upalpha}_{K}\cdot y^{\upalpha_{c}}_{K^{c}}y^{\bet}+y^{\upalpha}_{K}\cdot x^{\upalpha_{c}}_{K^{c}}x^{\bet})$
both belong to $\J.$ The proof is complete. 
  \end{proof}

\section{Proof of  Theorem \ref{T2}}\label{sec6}
\setcounter{equation}{0}
\renewcommand{\theequation}
{6.\arabic{equation}}
\setcounter{theorem}{0}
\renewcommand{\thetheorem}
{6.\arabic{theorem}}

\noindent Let $R:=\Fq[N_{1},N_{2},B_{0},B_{q-1}]$ and $R':=\Fq[N_{1},N_{2}]$. We have seen that $R=\Fq[2V]^{\OO\times\OO}$ and 
 $\{N_{1},N_{2},B_{0},B_{q-1}\}$ is a homogeneous system of parameters for $\Fq[2V]^{\OO}$. We define 
 $$\mathcal{M}:=\Big\{U_{12}^{i}\mid 0\leqslant i\leqslant \frac{q}{2}\Big\}\cup \Big\{B_{k}\mid 1\leqslant k\leqslant q-2\Big\}\cup\Big\{B_{i}\cdot B_{j}\mid 1\leqslant i,j\leqslant q-2\textrm{ and }i+j=q-1\Big\}.$$

 \begin{prop}\label{6.1}
For $1\leqslant k\leqslant q-2$, we have $B_{k}\cdot U_{12}=N_{2}\cdot B_{k+1}+N_{1}\cdot B_{k-1}\in \sum_{k=1}^{q-2} R\cdot B_{k}.$
\end{prop} 
 
  \begin{proof}
Indeed,  $
B_{k}\cdot U_{12}  =  (x_{1}^{k}x_2^{q-1-k}+y_{1}^{k}y_2^{q-1-k})(x_1y_2+x_2y_1)
 =   (x_1^{k+1}y_2x_2^{q-1-k}+y_2^{q-1-k}x_2y_1^{k+1})+(x_{1}^{k}x_2^{q-k}y_1+y_{1}^{k}y_2^{q-k}x_1)
=N_{2}\cdot B_{k+1}+N_{1}\cdot B_{k-1}\in \sum_{k=1}^{q-2} R\cdot  B_{k}.
$
\end{proof} 
 
 \begin{prop}\label{6.2}
$U_{12}^{\frac{q}{2}+1}\in \sum_{i=0}^{\frac{q}{2}} R'\cdot  U_{12}^{i}.$
\end{prop}

  \begin{proof} 
Note that $q=2^{s}$ with $s\geqslant 2$.  If $s=2$, then $q=4$.  It is easy to check that $U_{12}^{3}=U_{12}^{2}+N_{1}N_{2}U_{12}.$ This statement  follows in this special case.
Now we suppose $s\geqslant 3$ and define $$V_{j}:=x_1^{\frac{q}{2}+1-2j}y_2^{\frac{q}{2}+1-2j}+y_1^{\frac{q}{2}+1-2j}x_2^{\frac{q}{2}+1-2j}$$ for $j=0,1,2,\dots,\frac{q}{4}$. In particular, $V_{\frac{q}{4}}=U_{12}$ and $V_{\frac{q}{4}-1}=U_{12}^{3}=U_{12}^{2}+N_{1}N_{2}U_{12}$. Then for $i=0,1,2,\dots,\frac{q}{4}-2,$ we have
\begin{eqnarray*}
V_{j} & = & x_1^{\frac{q}{2}+1-2j}y_2^{\frac{q}{2}+1-2j}+y_1^{\frac{q}{2}+1-2j}x_2^{\frac{q}{2}+1-2j} \\
 & = & (x_1^{\frac{q}{2}-1-2j}y_2^{\frac{q}{2}-1-2j}+y_1^{\frac{q}{2}-1-2j}x_2^{\frac{q}{2}-1-2j})(x_1^{2}y_2^{2}+y_1^{2}x_2^{2})+\\
 &&(y_1^{\frac{q}{2}-1-2j}x_2^{\frac{q}{2}-1-2j}x_1^{2}y_2^{2}+x_1^{\frac{q}{2}-1-2j}y_2^{\frac{q}{2}-1-2j}y_1^{2}x_2^{2})\\
 &=& V_{j+1}U_{12}^{2}+(N_{1}N_{2})^{2}V_{j+2}.
\end{eqnarray*}
Thus, $V_{0}=V_{1}U_{12}^{2}+(N_{1}N_{2})^{2}V_{2}=(V_{2}U_{12}^{2}+(N_{1}N_{2})^{2}V_{3})U_{12}^{2}+(N_{1}N_{2})^{2}V_{2}=\cdots=U_{12}^{\frac{q}{2}}+f$, where $f\in\sum_{i<\frac{q}{2}}  R'\cdot U_{12}^{i}$.
Hence,  $U_{12}^{\frac{q}{2}+1}  =  (x_1^{\frac{q}{2}}y_2^{\frac{q}{2}}+x_2^{\frac{q}{2}}y_1^{\frac{q}{2}})(x_1y_2+x_2y_1)
=(x_1^{\frac{q}{2}+1}y_2^{\frac{q}{2}+1}+y_1^{\frac{q}{2}+1}x_2^{\frac{q}{2}+1})+(x_1^{\frac{q}{2}}x_2y_1y_2^{\frac{q}{2}}+
y_1^{\frac{q}{2}}y_2x_1x_2^{\frac{q}{2}})
=V_{0}+N_{1}N_{2}V_{1}=U_{12}^{\frac{q}{2}}+f,
$ for some $f\in\sum_{i<\frac{q}{2}}  R'\cdot U_{12}^{i}$. Therefore, $U_{12}^{\frac{q}{2}+1}\in \sum_{i=0}^{\frac{q}{2}} R'\cdot U_{12}^{i}$, as desired. 
\end{proof}

 \begin{lem}\label{6.3}
For any $n\in\N^{+}$, $v_{n}:=y_1^{n}x_{2}^{n}+x_1^{n}y_{2}^{n}\in \sum_{f\in\mathcal{M}} R\cdot f.$
\end{lem} 
 
  \begin{proof} We use  induction on $n$.
If $n=1$, then $v_{1}=U_{12}$ and the lemma follows immediately. Suppose $n\geqslant 2$, then
$v_{n}=(y_1^{n-1}x_{2}^{n-1}+x_1^{n-1}y_{2}^{n-1})(y_1x_{2}+x_1y_{2})+(x_1^{n-1}y_{2}^{n-1}y_1x_{2}+
y_1^{n-1}x_{2}^{n-1}x_1y_{2})=v_{n-1}U_{12}+N_{1}N_{2}v_{n-2}.$ By the induction hypothesis and Proposition \ref{6.2}, we have
$v_{n}\in \sum_{f\in\mathcal{M}} R\cdot f.$
\end{proof}

 \begin{prop}\label{6.4}
For $1\leqslant k\leqslant i\leqslant q-2$, we have $B_{k}\cdot B_{i}\in \sum_{f\in\mathcal{M}} R\cdot f.$ 
\end{prop} 
 
\begin{proof} If $k+i=q-1$, then $B_{k}\cdot B_{i}\in \mathcal{M}$.
Now we consider the case when $k+i<q-1$. Note that
\begin{eqnarray*}
B_{k}\cdot B_{i} & = & (x_1^{k}x_2^{q-1-k}+y_1^{k}y_2^{q-1-k})(x_1^{i}x_2^{q-1-i}+y_1^{i}y_2^{q-1-i}) \\
 & = & (x_1^{k+i}x_2^{2q-2-k-i}+y_1^{k+i}y_2^{2q-2-k-i})+(x_1^{k}x_2^{q-1-k}y_1^{i}y_2^{q-1-i}+y_1^{k}y_2^{q-1-k}x_1^{i}x_2^{q-1-i})\\
 &=& (B_{k+i}+y_1^{k+i}y_2^{q-1-k-i})x_{2}^{q-1}+(B_{k+i}+x_1^{k+i}x_2^{q-1-k-i})y_{2}^{q-1}+\\
 &&N_{1}^{k}N_{2}^{q-1-i}(x_{2}^{i-k}y_{1}^{i-k}+y_{2}^{i-k}x_{1}^{i-k})\\
 &=& B_{k+i}B_{0}+N_{2}^{q-1-k-i}(y_1^{k+i}x_2^{k+i}+x_1^{k+i}y_2^{k+i})+N_{1}^{k}N_{2}^{q-1-i}(x_{2}^{i-k}y_{1}^{i-k}+y_{2}^{i-k}x_{1}^{i-k}).
\end{eqnarray*}
By Lemma \ref{6.3}, $v_{k+i}$ and $v_{i-k}$ both belong to $\sum_{f\in\mathcal{M}} R\cdot f$, so does $B_{k}\cdot B_{i}$.
Similar arguments can be applied to the case when $k+i>q-1$.
\end{proof} 

Now we are ready to prove Theorem \ref{T2}.

  \begin{proof}[Proof of  Theorem \ref{T2}]
  Since $\mathbb{F}_{q}[2V]^{\OO}$ is Cohen-Macaulay and $\Fq[2V]^{\OO\times\OO}$ is a polynomial algebra, it follows  that 
  $\mathbb{F}_{q}[2V]^{\OO}$ is a free $\Fq[2V]^{\OO\times\OO}$-module of rank $2(q-1)=|\OO|$, see for example  \cite[Lemma 2.1]{Che14}. 
We observe  that $|\mathcal{M}|=2(q-1).$ Thus to prove Theorem \ref{T2}, we need only to show that 
for any $g\in\Fq[2V]^{\OO}$, we have $g\in \sum_{f\in\mathcal{M}} R\cdot f$.
By Example \ref{e1.5}, it is sufficient to show that $$U_{12}^{\upalpha_{0}}\cdot B_{1}^{\upalpha_{1}} \cdot B_{2}^{\upalpha_{2}}\cdots B_{q-2}^{\upalpha_{q-2}}\in \sum_{f\in\mathcal{M}} R\cdot f,$$
for any nonzero vector $\upalpha=(\upalpha_{0},\upalpha_{1},\dots,\upalpha_{q-2})\in \N^{q-1}$.
By Proposition \ref{6.1}, it suffices to show that 
  $U_{12}^{\upalpha_{0}}\in \sum_{f\in\mathcal{M}} R\cdot f$ 
  and $B_{1}^{\upalpha_{1}} \cdot B_{2}^{\upalpha_{2}}\cdots B_{q-2}^{\upalpha_{q-2}}\in \sum_{f\in\mathcal{M}} R\cdot f$
  for any $\upalpha_{0}\in\N$
  and any $(\upalpha_{1},\dots,\upalpha_{q-2})\in \N^{q-2}$. The two cases follow from 
Proposition \ref{6.2} and Proposition \ref{6.4} respectively. 
  \end{proof}

 \section{More Examples and Proof of Corollary \ref{noethernumber}} \label{sec7}
\setcounter{equation}{0}
\renewcommand{\theequation}
{7.\arabic{equation}}
\setcounter{theorem}{0}
\renewcommand{\thetheorem}
{7.\arabic{theorem}}

 \begin{exam}\label{e7.1}
 {\rm ($\mathbb{F}_{q}[3V]^{\OO}$)
 Note that in this case,
$1\leqslant|I|\leqslant m-1=2$ and $1\leqslant|J|\leqslant m-1=2$.
 Since $s\geqslant 2$, $q=2^{s}\geqslant 4$ and $q-1\geqslant 3$. Thus $|I|-|J|=0$ in $\mathcal{D}$. It follows that either
  $|I|=|J|=1$ or  $|I|=|J|=2$. Since $I<J$, we must have  $|I|=|J|=1$.
 Hence, $\mathcal{U}=\mathcal{D}$. Theorem \ref{T1} tells us that  $\mathbb{F}_{q}[3V]^{\OO}$ is generated by the following 
 invariants:
\begin{eqnarray*}
\mathcal{N} & = & \Big\{N_1,N_2,N_{3}\Big\}\\
\mathcal{U} &=& \Big\{U_{12},U_{13},U_{23}\Big\}\\
\mathcal{B} &=&\Big\{x_1^{k}x_{2}^{t}x_3^{q-1-k-t}+y_1^{k}y_{2}^{t}y_3^{q-1-k-t}\mid 0\leqslant k,t\leqslant q-1\Big\}.
\end{eqnarray*}  
It is not hard to see that $|\mathcal{B}|=q+(q-1)+(q-2)+\cdots+2+1=\frac{q(q+1)}{2}$. Thus 
$|\mathcal{S} |=|\mathcal{N} |+|\mathcal{U} |+|\mathcal{B} |=\frac{q(q+1)}{2}+6.$
For instance, when $q=4$, $|\mathcal{S}|=16$ and  when $q=8$,  $|\mathcal{S}|=42$.
 \hbo} \end{exam}

 \begin{exam}
 {\rm ($m\geqslant 4$)
For $\mathbb{F}_{q}[4V]^{\OO}$,  there are no $d_{I,J}\in \mathcal{D}$ such that 
$q-1=|J|-|I|$. However, we have one element $x_{1}x_{2}y_{3}y_{4}+y_{1}y_{2}x_{3}x_{4}\in \mathcal{D}-\mathcal{U}$. 
When $m\geqslant 5$,  for $\mathbb{F}_{q}[mV]^{\OO}$,  there exists $d_{I,J}\in \mathcal{D}$ such that 
$q-1=|J|-|I|$. For example, we take $q=2^{2}=4$ and $m=5$. Then 
$x_{1}y_{2}y_{3}y_{4}y_{5}+y_{1}x_{2}x_{3}x_{4}x_{5}\in \mathcal{D}.$ 
  \hbo} \end{exam}

 \begin{proof}[Proof of Corollary \ref{noethernumber}]
Note that $q=2^{s}\geqslant 4$.   
Proposition \ref{2.3}, Example \ref{e1.5} and Example \ref{e7.1} show that 
$\bet_{mV}(\OO)=q-1$ for $m=1,2,3$ respectively. 
Now we suppose $m>3$.
If $m\leqslant q-1$, then any generator from $\mathcal{B}$ can make $\bet_{mV}(\OO)=q-1$ holds.
If $m>q-1$ and $m=2n$ is even, then $d_{I,J}\in \mathcal{D}$ with $I=\{1,2,\dots,n\}$ and $J=\{n+1,n+2,\dots,m\}$, implies that $\bet_{mV}(\OO)=m$. If $m>q-1$ and $m$ is odd, then $m-(q-1)$ is even. We may assume that $m-(q-1)=2n$.
Then $d_{I,J}\in \mathcal{D}$ with $I=\{1,2,\dots,n\}$ and $J=\{n+1,n+2,\dots,2n,2n+1,\dots,m\}$, implies that $\bet_{mV}(\OO)=m$.
 \end{proof}

 \section{Proof of  Theorem \ref{T3}}\label{sec8}
\setcounter{equation}{0}
\renewcommand{\theequation}
{8.\arabic{equation}}
\setcounter{theorem}{0}
\renewcommand{\thetheorem}
{8.\arabic{theorem}}

\begin{proof}[Proof of Theorem \ref{T3}]
Let $\mathcal{D}_{1}=\{d_{I,J}:|J|=|I|\}$ and $\mathcal{D}_{2}=\{d_{I,J}:|J|=|I|+q-1\}$. Then $\mathcal{U}\subseteq \mathcal{D}_{1}$ and $\mathcal{D}=\mathcal{D}_{1}\cup \mathcal{D}_{2}$. 
By Theorem \ref{T1}, it is sufficient to show that any element in $\mathcal{D}$ is contained in $\mathfrak{A}$, the ideal generated by $\mathcal{N}\cup\mathcal{U}\cup\mathcal{B}$ in $\Fq[mV]$. Our arguments will be completed by showing two subcases: $\mathcal{D}_{1}\subseteq \mathfrak{A}$ and $\mathcal{D}_{2}\subseteq \mathfrak{A}$.

\textsc{Subcase 1}: For any $d_{I,J}\in \mathcal{D}_{1}$, we suppose the degree of $d_{I,J}$ is $2n$. We use induction on $n$. When $n=1$, we may write $d_{I,J}=x_{i}y_{j}+y_{i}x_{j}$ with $1\leqslant i<j\leqslant m$. Thus $d_{I,J}=U_{ij}\in \mathfrak{A}$.
Now consider any $n>1$. Let $i_{1}$ be the minimal integer in $I$ and 
$j_{1}$ be the minimal integer in $J$. Note that $i_{1}<j_{1}$, then $d_{I,J}=x_{I}\cdot y_{J}+y_{I}\cdot x_{J}=
x_{i_{1}}y_{j_{1}}(x_{I-\{i_{1}\}}\cdot y_{J-\{j_{1}\}})+y_{i_{1}}x_{j_{1}}(y_{I-\{i_{1}\}}\cdot x_{J-\{j_{1}\}})=
(U_{i_{1}j_{1}}+y_{i_{1}}x_{j_{1}})(x_{I-\{i_{1}\}}\cdot y_{J-\{j_{1}\}})+y_{i_{1}}x_{j_{1}}(y_{I-\{i_{1}\}}\cdot x_{J-\{j_{1}\}})
\equiv_{\mathfrak{A}}(y_{i_{1}}x_{j_{1}})\cdot d_{I-\{i_{1}\},J-\{j_{1}\}}$. Since $d_{I-\{i_{1}\},J-\{j_{1}\}}$ has degree $2(n-1)$, the induction hypothesis implies that $d_{I-\{i_{1}\},J-\{j_{1}\}}\in \mathfrak{A}$. Thus $d_{I,J}\in \mathfrak{A}$.

\textsc{Subcase 2}: For any $d_{I,J}\in \mathcal{D}_{2}$, we let  $I=\{i_{1},\dots,i_{k}\}$ and $J=\{j_{1},\dots,j_{k},j_{k+1},\dots,j_{k+q-1}\}$, where $1\leqslant i_{1}<\cdots<i_{k}<j_{1}<\cdots<j_{k+q-1}\leqslant m$.
We may write 
\begin{eqnarray*}
d_{I,J}&=&\prod_{i=i_{1}}^{i_{k}} x_{i}\cdot \prod_{j=j_{1}}^{j_{k}} y_{j} \cdot \prod_{j=j_{k+1}}^{j_{k+q-1}} y_{j}+\prod_{i=i_{1}}^{i_{k}} y_{i}\cdot \prod_{j=j_{1}}^{j_{k}} x_{j} \cdot \prod_{j=j_{k+1}}^{j_{k+q-1}} x_{j}\\
 & = & \prod_{i=i_{1}}^{i_{k}} x_{i}\cdot \prod_{j=j_{1}}^{j_{k}} y_{j} \cdot \prod_{j=j_{k+1}}^{j_{k+q-1}} y_{j}+\prod_{i=i_{1}}^{i_{k}} y_{i}\cdot \prod_{j=j_{1}}^{j_{k}} x_{j} \cdot (B_{\overline{1}}+\prod_{j=j_{k+1}}^{j_{k+q-1}} y_{j})\\
 &\equiv& \Big(\prod_{i=i_{1}}^{i_{k}} x_{i}\cdot \prod_{j=j_{1}}^{j_{k}} y_{j}+\prod_{i=i_{1}}^{i_{k}} y_{i}\cdot \prod_{j=j_{1}}^{j_{k}} x_{j}\Big) \cdot \prod_{j=j_{k+1}}^{j_{k+q-1}} y_{j}, \quad \textrm{ mod }(\mathfrak{A}).
\end{eqnarray*}
Since $\prod_{i=i_{1}}^{i_{k}} x_{i}\cdot \prod_{j=j_{1}}^{j_{k}} y_{j}+\prod_{i=i_{1}}^{i_{k}} y_{i}\cdot \prod_{j=j_{1}}^{j_{k}} x_{j}\in
\mathcal{D}_{1}$, it follows from the first subcase that $d_{I,J}\in \mathfrak{A}$. Therefore,
$\mathcal{D}_{2}\subseteq \mathfrak{A}$, completing the proof. 
 \end{proof}

  \section{Remarks on $\Oo$ and $\Ir$}\label{sec9}
\setcounter{equation}{0}
\renewcommand{\theequation}
{9.\arabic{equation}}
\setcounter{theorem}{0}
\renewcommand{\thetheorem}
{9.\arabic{theorem}}

\noindent In this last section, we discuss $\Oo$ and $\Ir$. To our knowledge, there are no suitable references concerning a detailed description for  generators of the group $\Oo$ in terms of  matrix language. 
 
First of all, we need to find out a set of generators for $\Oo$, which is more complicated than the case of $\OO$.
Let  $\begin{pmatrix}
    a  & b   \\
     c &  d
\end{pmatrix}\in \Oo$ be any element. By the definition, we have
\begin{eqnarray*}
\begin{pmatrix}
    a  & b   \\
     c &  d
\end{pmatrix}\begin{pmatrix}
   w  & 1   \\
     0&  w
\end{pmatrix}\begin{pmatrix}
    a  & c   \\
     b &  d
\end{pmatrix}-\begin{pmatrix}
   w  & 1   \\
     0&  w
\end{pmatrix}&=&\begin{pmatrix}
   aw  & a+bw   \\
    cw& c+dw 
\end{pmatrix}\begin{pmatrix}
    a  & c   \\
     b &  d
\end{pmatrix}-\begin{pmatrix}
   w  & 1   \\
     0&  w
\end{pmatrix}\\
&=&\begin{pmatrix}
  a^{2}w+b^{2}w+w+ab    & acw+bdw+ad+1   \\
  acw+bdw+bc   &    c^{2}w+d^{2}w+w+cd
\end{pmatrix}
\end{eqnarray*}  
is an alternating matrix, i.e.,   
\begin{eqnarray}
  a^{2}w+b^{2}w+w+ab &=& 0 \label{eq9.1}\\
 c^{2}w+d^{2}w+w+cd &=&0 \label{eq9.2}\\
ad+bc+1 &=& 0. \label{eq9.3}
\end{eqnarray}

\textsc{Case} 1. Suppose that  $a=0$, it follows from (\ref{eq9.1}) and Eq. (\ref{eq9.3})  that $b^{2}=1$ and $c=b$. 
Since $b\in \mathbb{F}_{q}^{\times}$ and the order of $\mathbb{F}_{q}^{\times}$ is odd, we have
$b=c=1$.
It follows from  (\ref{eq9.2}) that 
$d^{2}w=d.$ If $d=0$, we obtain an orthogonal matrix 
$$\sig:=\begin{pmatrix}
     0 & 1   \\
     1 &  0
\end{pmatrix}.$$
 If $d\neq 0$, then $d=w^{-1}$ and we have another orthogonal matrix
 $$\uptau_{0}:=\begin{pmatrix}
     0 & 1   \\
     1 &  w^{-1}
\end{pmatrix}.$$
  
\textsc{Case} 2. 
Suppose  $a\neq0$,  it follows from  (\ref{eq9.3})  that $d=\frac{bc+1}{a}$. Combining  (\ref{eq9.2}) and (\ref{eq9.1}), we have
\begin{equation}
\label{eq9.4}
 a^{2}w+c^{2}w+w+ac = 0.
\end{equation}
Adding  (\ref{eq9.4}) to  (\ref{eq9.1}), we obtain 
$(b^{2}+c^{2})w= a(b+c).$
 If $b=c$, we have a family of orthogonal matrices
$$\uptau_{a}:=\begin{pmatrix}
   a   & b   \\
    b  &  a+bw^{-1}
\end{pmatrix},
  $$ and if $b\neq c$ we have
$$\varepsilon_{a}:=\begin{pmatrix}
   a   & b   \\
    aw^{-1}+b  &  a
\end{pmatrix}
  $$ 
where $b$ is defined by $a^{2}w+b^{2}w+w+ab = 0$. Note that $\varepsilon_{b}=\sig\cdot \uptau_{a}$ for all $a\in \mathbb{F}_{q}$. Thus $\Oo$  consists of the following matrices:
$\Big\{1,\sig,\uptau_{a},\sig\cdot\uptau_{a}\mid a\in \mathbb{F}_{q} \Big\}.$

Secondly, we consider  the invariant ring $\Ir$.
Magma calculations \cite{BCP97} suggest that $\Fq[V]^{\Oo}=\mathbb{F}_{q}[x,y]^{\Oo}$
might be a polynomial algebra with two generators $Q$ and $E$, of degrees $2$ and $q+1$ respectively.
 We define 
 \begin{eqnarray}
E & := & xy^{q}+x^{q}y \\
Q & := & {\rm Tr}^{\Oo}(x^{2}).
\end{eqnarray}
We \textit{claim} that $\mathbb{F}_{q}[V]^{\Oo}=\mathbb{F}_{q}[E,Q].$
Since $|\Oo|=\textrm{deg}(E)\cdot\textrm{deg}(Q)$, we only need to show that the Jacobian determinant
$$\textrm{det}\begin{pmatrix}
  \frac{\partial E}{\partial x}    &    \frac{\partial E}{\partial y}  \\
  \frac{\partial Q}{\partial x}    &    \frac{\partial Q}{\partial y}
\end{pmatrix}\neq 0,$$
by Kemper \cite[Proposition 16]{Kem96}. 
  We write $Q=x^{2}+uxy+vy^{2}$ for some $u,v\in \mathbb{F}_{q}$. Since $Q$ is  an $\Oo$-invariant, a simple computation shows that $u\neq 0$. Thus 
  $$\textrm{det}\begin{pmatrix}
  \frac{\partial E}{\partial x}    &    \frac{\partial E}{\partial y}  \\
  \frac{\partial Q}{\partial x}    &    \frac{\partial Q}{\partial y}
\end{pmatrix}=\begin{pmatrix}
     y^{q} & x^{q}   \\
    uy  &  ux
\end{pmatrix}=u\cdot E\neq 0,$$ which shows the claim.  For the case $m=2$ and some small $q$, Magma calculations \cite{BCP97} suggest that $\Fq[2V]^{\Oo}$ can be generated by
$q+5$ invariants: $N'_1,N'_2,U_{12}$, and $B'_{k}$
for $0\leqslant k\leqslant q+1$.  
This evidence suggests that the approach used in the calculation of $\IR$ might be applied to 
study the invariant ring $\Ir$.

 \section*{Acknowledgments}
  
 \noindent This research was done during the author's visit at Queen's University at Kingston, Canada in 2014--2016. The author would like to thank David L. Wehlau for his support,  conversations and careful reading the draft of this paper.  
 The author thanks the anonymous referee for his/her helpful comments.
 This work was partially supported by the Fundamental Research Funds for the Central Universities (2412017FZ001), NSF of China (11401087),  and NSERC.
The symbolic computation language MAGMA \cite{BCP97} (http://magma.maths.usyd.edu.au/)  was very helpful.


\raggedright
\end{document}